\documentclass[10pt]{amsart}
\usepackage{amsmath,amssymb}
\usepackage{enumerate}
\usepackage{color}
\usepackage{hyperref}
\usepackage{natbib}

\usepackage{algorithm}
\usepackage{algpseudocode}
\usepackage{graphicx}
\usepackage{dsfont}

\newtheorem{theorem}{Theorem}[section]
\newtheorem{proposition}[theorem]{Proposition}

\newtheorem{corollary}[theorem]{Corollary}
\newtheorem{definition}[theorem]{Definition}
\newtheorem{remark}[theorem]{Remark}
\newtheorem{example}[theorem]{Example}

\newcommand{\R}{\mathbb{R}}

\newcommand{\N}{\mathbb{N}}

\renewcommand{\epsilon}{\varepsilon}

\begin{document}

\title[An OT characterization of convex order]{An optimal transport based characterization of convex order}

\date{\today}

\author{Johannes Wiesel}
\address{Johannes Wiesel\newline
Columbia University, Department of Statistics\newline
1255 Amsterdam Avenue\newline
New York, NY 10027, USA}
\email{johannes.wiesel@columbia.edu}
\author{Erica Zhang}
\address{Erica Zhang\newline
Columbia University, Department of Statistics\newline
1255 Amsterdam Avenue\newline
New York, NY 10027, USA}
\email{yz4232@columbia.edu}

\begin{abstract}
For probability measures $\mu,\nu$ and $\rho$ define the cost functionals 
\begin{align*}
C(\mu,\rho):=\sup_{\pi\in \Pi(\mu,\rho)} \int \langle x,y\rangle\, \pi(dx,dy),\quad C(\nu,\rho):=\sup_{\pi\in \Pi(\nu,\rho)} \int \langle x,y\rangle\, \pi(dx,dy),
\end{align*}
where $\langle\cdot, \cdot\rangle$ denotes the scalar product and $\Pi(\cdot,\cdot)$ is the set of couplings.
We show that two probability measures $\mu$ and $\nu$ on $\R^d$ with finite first moments are in convex order (i.e. $\mu\preceq_c\nu$) iff
$C(\mu,\rho)\le C(\nu,\rho)$
holds for all probability measures $\rho$ on $\R^d$ with bounded support. This generalizes a result by Carlier. Our proof relies on a quantitative bound for the infimum of $\int f\,d\nu -\int f\,d\mu$ over all $1$-Lipschitz functions $f$, which is obtained through optimal transport duality and Brenier's theorem. Building on this result, we derive new proofs of well-known one-dimensional characterizations of convex order. We also describe new computational methods for investigating convex order and applications to model-independent arbitrage strategies in mathematical finance.
\end{abstract}

\thanks{JW acknowledges support by NSF Grant DMS-2205534. Part of this research was performed while JW was visiting the Institute for Mathematical and Statistical Innovation (IMSI), which is supported by the National Science Foundation (Grant No. DMS-1929348). JW thanks Beatrice Acciaio, Guillaume Carlier, Max Nendel, Gudmund Pammer and Ruodu Wang for helpful discussions.
EZ acknowledges support through the summer internship program of the Columbia university statistics department.}
\maketitle

\section{Introduction and main result}

Fix two probability measures $\mu,\nu\in \mathcal{P}(\R^d)$ with $$\int |x|\,\mu(dx)<\infty, \quad \int |y|\,\nu(dy)<\infty.$$
Recall that $\mu$ and $\nu$ are in convex order (denoted by $\mu\preceq_c \nu$) iff
\begin{align*}
\int f\,d\mu\le \int f\,d\nu \qquad \text{for all convex functions } f:\R^d\to \R.
\end{align*}
As any convex function is bounded from below by an affine function, the above integrals take values in $(-\infty, \infty]$.
The notion of convex order is very well studied, see e.g. \cite{ross1996stochastic, muller2002comparison,shaked2007stochastic, arnold2012majorization} and the references therein for an overview. It plays a pivotal role in mathematical finance since \cite{strassen1965existence} established that $\mu\preceq_c\nu$ if and only if $\mathcal{M}(\mu,\nu)$ --- the set of martingale laws on $\R^d\times \R^d$ with marginals $\mu$ and $\nu$ --- is non-empty. This result is also the reason why convex order  has taken the center stage in the field of martingale optimal transport, see e.g. \cite{galichon2014stochastic, beiglbock2013model, beiglbock2015complete, de2019irreducible, obloj2017structure, guo2019computational, alfonsi2017sampling, alfonsi2020sampling,alfonsi2020squared, jourdain2022martingale, pietro} and the references therein. Furthermore, convex order plays a pivotal role in dependence modelling and risk aggregation, see e.g. \cite{tchen1980inequalities, ruschendorf2002variance, wang2011complete, embrechts2013model, bernard2017value}.   \\
While there is an abundance of explicit characterizations of convex order available in one dimension (i.e. $d=1$) -- see e.g. \cite[Chapter 3]{shaked2007stochastic}) --- the case $d>1$ seems to be less studied to the best of our knowledge. The main goal of this article is to fill this gap: we discuss a  characterization of convex order, that holds in general dimensions, and is based on the theory of optimal transport (OT). Optimal transport goes back to the seminal works of \cite{monge1781memoire} and \cite{kantorovich1942translocation}. It is concerned with the problem of transporting probability distributions in a cost-optimal way. We refer to \cite{rachev1998mass} and \cite{villani2003topics, villani2008optimal} for an overview. For this paper we only need a few basic concepts from OT. Most importantly we will need the cost functionals
\begin{align*}
C(\mu,\rho):=\sup_{\pi\in \Pi(\mu,\rho)} \int \langle x,y\rangle\, \pi(dx,dy),\qquad C(\nu,\rho):=\sup_{\pi\in \Pi(\nu,\rho)} \int \langle x,y\rangle\, \pi(dx,dy).
\end{align*}
Here $\Pi(\mu,\nu)$ denotes the set of probability measures on $\R^d\times \R^d$ with marginals $\mu$ and $\nu$. Our main result is the following:
\begin{theorem}\label{thm:main2}
Assume that $\mu,\nu\in \mathcal{P}(\R^d)$ have finite first moments.
Then 
\begin{align}\label{eq:main}
\inf_{f\in \mathcal{C}^1(\R^d)}\left(\int f\,d\nu - \int f\,d\mu \right)=\inf_{\rho\in \mathcal{P}^1(\R^d)} \left( C(\nu,\rho)-C(\mu,\rho)\right),
\end{align}
where
\begin{align*}
\mathcal{P}^1(\R^d):=\{\rho\in \mathcal{P}(\R^d): \ \mathrm{supp}(\rho)\subseteq B_1(0)\}
\end{align*}
and
\begin{align*}
\mathcal{C}^1(\R^d):=\{f:\R^d\to \R \ \mathrm{ convex}, 1\text{-}\mathrm{Lipschitz}\}.
\end{align*}
\end{theorem}

Theorem \ref{thm:main2}  states that convex order of $\mu$ and $\nu$ is equivalent to an order relation $C(\cdot, \cdot)$ on the space of probability measures. Contrary to standard characterizations of convex order using potential functions or cdfs,  it holds in any dimension and can be seen as a natural generalization of the following result:
\begin{corollary}\label{thm:main}
Denote the 2-Wasserstein metric by
\begin{align*}
\mathcal{W}_2 (\mu,\nu):= \inf_{\pi\in \Pi(\mu,\nu)} \sqrt{\int |x-y|^2\,\pi(dx,dy)}.
\end{align*}
If $\mu$ and $\nu$ have finite second moment, then they are in convex order if and only if 
\begin{align}\label{eq:main3}
 \mathcal{W}_2(\nu, \rho)^2- \mathcal{W}_2(\mu,\rho)^2\le \int |y|^2\,\nu(dy)- \int |x|^2\,\mu(dx)
\end{align}
holds for all probability measures $\rho$ on $\R^d$ with bounded support.
\end{corollary}

Corollary \ref{thm:main} itself has an interesting history. To the best of our knowledge, it was first stated in \cite{carlier2008remarks} for compactly supported measures $\mu, \nu$. His proof relies on a well-known connection between convex functions and OT for the squared Euclidean distance called Brenier's theorem (see \cite{brenier1991polar, ruschendorf1990characterization}) together with a certain probabilistic first-order condition, see \cite[Proposition 1]{carlier2008remarks}. We emphasize here, that contrary to the setting of Brenier's theorem, no assumptions on the probability measures $\mu$ and $\nu$ except for the compact support condition are made; in particular there is no need to assume that these are absolutely continuous wrt. the Lebesgue measure.

 Interestingly, Carlier's result does not seem to be very well-known in the literature on stochastic order. We conjecture that this is mainly due to his use of the french word ``balay\'ee" instead of convex order, so that the connection is not immediately apparent.  For this reason, one aim of this note is to popularize Carlier's result, making it accessible to a wider audience, while simultaneously showcasing potential applications. As it turns out, Corollary \ref{thm:main} is at least partially known to the mathematical finance community: indeed, the ``only if" direction of Corollary \ref{thm:main} was rediscovered in \cite[Equation (2.2)]{alfonsi2020squared} for (not necessarily compactly supported) probability measures $\mu,\nu$ with finite second moments.
 
Theorem \ref{thm:main2} differs from Carlier's work in three aspects: first, as the convex order is classically embedded in $\mathcal{P}_1(\R^d)$ and does not require moments of higher order or compact support assumptions (see e.g. \cite{nendel2020note}), Theorem \ref{thm:main2} is simultaneously more concise and arguably more natural than Corollary \ref{thm:main}. Second, our proof of Theorem \ref{thm:main2} (and thus also Corollary \ref{thm:main}) follows a different route than Carlier's original proof, who argues purely on the space probability measures (i.e. the ``primal side" in optimal transport). Instead, we combine Brenier's theorem with the theory of the classical optimal transport duality. Lastly, we discuss three implications of Theorem \ref{thm:main2}: we first give a proof of a characterization of convex order in one dimension through quantile functions. Then we use Theorem \ref{thm:main2} to derive new computational methods for testing convex order between $\mu$ and $\nu$. For the computation we exploit state of the art computational OT methods, which are efficient for potentially high-dimensional problems. These have recently seen a spike in research activity. We refer to \cite{peyre2019computational} for an overview. Finally we discuss applications of Theorem \ref{thm:main2} to the theory of so-called model-independent arbitrages, see \cite[Definition 1.2]{acciaio2013model}.

This article is structured as follows: in Section \ref{sec:dis} we state examples and consequences of Theorem \ref{thm:main2}. In particular we connect it to some well-known results in the theory of convex order. The proof of the main results is given in Section \ref{sec:proof}. Sections \ref{sec:num} and \ref{sec:arbitrage} discuss numerical and mathematical finance applications of Theorem \ref{thm:main2} respectively. Remaining proofs are collected in Section \ref{sec:rem}.

\section{Discussion and consequences of main results} \label{sec:dis}

To sharpen intuition, let us first discuss the case $d=1$. By Theorem \ref{thm:main2} we can obtain a new proof of a well-known representation of convex order on the real line, see e.g. \cite[Theorem 3.A.5]{shaked2007stochastic}. Here we denote the quantile function of a probability measure $\mu$ by 
\begin{align*}
F_\mu^{-1}(x):=\inf\{y\in \R: \ \mu((-\infty, y])\ge x\}.
\end{align*}

\begin{corollary}\label{cor:d1}
For $d=1$ we have 
\begin{align*}
\mu\preceq_c\nu &\Leftrightarrow \int_{0}^x [F^{-1}_\mu(y)-F_\nu^{-1}(y)]\,dy\ge 0.
\end{align*}
for all $x\in [0,1]$, with equality for $x=1$.
\end{corollary}
The proofs of all results of this section are collected in Section \ref{sec:rem}. We continue with general $d\in \N$ and give a geometric interpretation of Corollary \ref{thm:main} by restating it as follows: $\mu\preceq_c \nu$ holds iff
\begin{align}\label{eq:main2}
\mathcal{W}_2(\nu,\rho)^2-\mathcal{W}_2(\mu,\rho)^2\le \mathcal{W}_2(\nu, \delta_z)^2 -\mathcal{W}(\mu, \delta_z)^2
\end{align}
for all $\rho\in \mathcal{P}(\R^d)$ with bounded support, where $\delta_z$, $z\in \R^d$ is a Dirac measure. Indeed, varying $\rho$ over Dirac measures in \eqref{eq:main3} implies that the means of $\mu$ and $\nu$ have to be equal; equation \eqref{eq:main2} then follows from simple algebra. This implies in particular that the difference between squared Wasserstein cost from $\nu$ and $\mu$ to $\rho$ is maximised at the point masses. Lastly, Theorem \ref{thm:main2} can also be reformulated as: $\mu \preceq_c \nu$ iff
\begin{align*}
\sup_{\pi\in \Pi(\mu,\rho)} \int \langle x, z\rangle\,\pi(dx,dz) \le \sup_{\pi\in \Pi(\nu,\rho)} \int \langle y, z\rangle\,\pi(dy,dz),
\end{align*}
i.e. for any $\rho\in \mathcal{P}(\R^d)$ with bounded support, the maximal covariance between $\mu$ and $\rho$ is less than the one between $\nu$ and $\rho$. This provides a natural intuition for a classical pedestrian description of convex order, namely that ``$\nu$ being more spread out than $\mu$".

We next give a simple example for Corollary \ref{thm:main}.

\begin{example}
Let us take $\mu=\delta_0$ and $\nu$ with mean zero. Now, recalling \eqref{eq:was} and bounding $\mathcal{W}_2(\nu, \rho)$ from above by choosing the product coupling, we obtain that for any $\rho$ with finite second moment
\begin{align*} 
 \mathcal{W}_2(\nu, \rho)^2- \mathcal{W}_2(\mu,\rho)^2 &=  \mathcal{W}_2(\nu, \rho)^2- \int |x|^2\,\rho(dx)\\
&\le  \int |y|^2 \nu(dy) -\int 2\langle x, y\rangle \,\nu(dx)\rho(dy)\\
&=\int |y|^2 \nu(dy)\\
&= \int |y|^2\,\nu(dy)- \int |x|^2\,\mu(dx).
\end{align*}
In conclusion we recover the well-known fact $\delta_0\preceq_c \nu$.
\end{example}

We now state two direct corollaries of Corollary \ref{thm:main}. We consider the cost $c(x,y):=|x-y|^2/2$ and recall that a function $f$ is $c$-concave, if 
\begin{align*}
f(x)=\inf_{y\in \R^d} (g(y)-c(x,y))
\end{align*}
for some function $g:\R^d\to \R$. We then have the following:

\begin{corollary}\label{cor:c-concave}
We have
\begin{align*}
\int g\,d\nu \le \int g\,d\mu  \qquad \text{for all } c\text{-concave functions } g:\R^d\to \R
\end{align*}
if and only if
\begin{align*}
 \mathcal{W}_2(\nu, \rho)^2\le \mathcal{W}_2(\mu,\rho)^2\qquad \text{for all }\rho\in\mathcal{P}(\R^d)\text{ with compact support}.
\end{align*}
\end{corollary}

Corollary \ref{thm:main} also directly implies the following well-known result:
\begin{corollary}\label{cor:easy}
If $\mu\preceq_c \nu$ then $$\mathcal{W}_2(\mu,\nu)^2\le \int |y|^2\,\nu(dy)-\int|x|^2\,\mu(dx).$$
In particular $\mu\preceq_c \nu$ implies
\begin{align*}
\ \sup_{\pi\in \Pi(\mu,\nu)} \int \langle x,y\rangle \,\pi(dx,dy)\ge \int |x|^2\,\mu(dx).
\end{align*}
\end{corollary}

\section{Proof of Theorem \ref{thm:main2}}\label{sec:proof}

Let us start by setting up some notation. We denote the scalar product on $\R^d$ by $\langle \cdot, \cdot \rangle$. We write $|\cdot|$ for the Euclidean norm on $\R^d$. The ball in $\R^d$ around $x$ of radius $r>0$ will be denoted by $B_r(x)$. We write $\nabla f(x)$ for the derivative of a function $f:\R^d \to \R$ at a point $x\in \R^d$. We denote the $d$-dimensional Lebesgue measure by $\lambda$.

In order to keep this article self-contained, we summarise some properties of optimal transport at the beginning of this section, and refer to \cite[Chapter 2.1]{villani2003topics} for a more detailed treatment.\\
By definition we have for any $\rho\in \mathcal{P}_2(\R^d)$ that
\begin{align}\label{eq:was}
\mathcal{W}_2(\mu,\rho)^2=\int |x|^2\,\mu(dx)+\int |y|^2\,\rho(dy) -2\sup_{\pi\in \Pi(\mu,\rho)} \int \langle x,y\rangle\, \pi(dx,dy).
\end{align}
In this section we thus (re-)define the cost function $c(x,y):= \langle x, y \rangle$ and recall that the convex conjugate $f^*:\R^d\to \R\cup \{\infty\}$ of a function $f:\R^d\to \R$ is given by
\begin{align*}
f^*(y):=\sup_{x\in \R^d} \left( \langle y, x\rangle - f(x)\right).
\end{align*}
The subdifferential of a proper convex function $f:\R^d \to \R\cup\{\infty\}$ is defined as
\begin{align*}
\partial f(x):=\{y\in \R^d:\ f(x')-f(x)\ge \langle y, x'-x\rangle \text{ for all }x'\in \R^d\}.
\end{align*}
It is non-empty if $x$ belongs to the interior of the domain of $f$. We have
\begin{align}\label{eq:subdiff}
f(x)+f^*(y) -\langle x,y\rangle=0 \quad \Leftrightarrow \quad y\in \partial f(x).
\end{align}
Lastly we recall the duality
\begin{align}\label{eq:duality}
\begin{split}
C(\mu,\rho)&=\sup_{\pi\in \Pi(\mu,\rho)} \int \langle x,y\rangle\, \pi(dx,dy)\\
&= \inf_{f\oplus g \ge c} \int f\,d\mu+\int g\,d\rho\\
&= \inf_{f\oplus g \ge c, \ f,g \text{ proper, convex }} \int f\,d\mu+\int g\,d\rho
\end{split}
\end{align}
and the existence of an optimal pair $(f, f^*)$ of (lower semicontinuous, proper) convex conjugate functions. Replacing $\mu$ by $\nu$ in the display above, we obtain a similar duality for $C(\nu, \rho)$.

\subsection{Proof of Theorem \ref{thm:main2}: the equivalent case}

We first prove Theorem \ref{thm:main2} for measures $\mu,\nu$, which are equivalent to the $d$-dimensional Lebesgue measure $\lambda$, i.e. $\mu,\nu\sim\lambda$. As $\mu,\nu\in \mathcal{P}_1(\R^d)$, the domain of the optimising potential $f$ for $C(\mu,\rho)$ (resp. $C(\nu,\rho)$) is $\R^d$ in this case. Recall furthermore that
\begin{align}\label{eq:subdiff}
\partial f(x)=\overline{\text{Conv}}(\lim_{x_k\to x}\nabla f(x_k)),
\end{align}
see e.g. \cite[2.1.3.3)]{villani2003topics}.
We write $$\|\partial f \|_\infty:= \sup_{x\in \R^d}\sup_{y\in \partial f(x)} |y|.$$ \\
We now prove Theorem \ref{thm:main2} when $\mu,\nu\sim\lambda$:

\begin{proposition}\label{prop:1}
Assume $\mu,\nu\in \mathcal{P}_1(\R^d)$, $\mu,\nu\sim \lambda$. Recall 
\begin{align*}
\mathcal{P}^1(\R^d)=\{\rho\in \mathcal{P}(\R^d): \ \mathrm{supp}(\rho)\subseteq B_1(0)\}
\end{align*}
as well as the 1-Lipschitz convex functions
\begin{align*}
\mathcal{C}^1(\R^d)=\{f:\R^d\to \R \ \mathrm{ convex}, \|\partial f\|_{\infty}\le 1\}.
\end{align*}
Then we have
\begin{align*}
\inf_{f\in \mathcal{C}^1(\R^d)}\left(\int f\,d\nu - \int f\,d\mu \right)=\inf_{\rho\in \mathcal{P}^1(\R^d)} \left( C(\nu,\rho)-C(\mu,\rho)\right).
\end{align*}
\end{proposition}

\begin{proof}
As $\mu,\nu$ have finite first moment and $\rho$ is compactly supported, $|C(\mu,\rho)|,|C(\nu, \rho)|<\infty$ follows from H\"older's inequality.
We now fix $\rho\in \mathcal{P}^1(\R^d)$ and take an optimal convex pair $(\hat{f},\hat{g})$ in \eqref{eq:duality} for $C(\nu,\rho)$. %
Next we apply Brenier's theorem in the form of \cite[Theorem 2.12]{villani2003topics}), which states that $\rho=\nabla \hat{f}(x)_*\nu$.\footnote{We note that the result is stated under the additional requirement that $\nu, \rho \in \mathcal{P}_2(\R^d)$. However as $\rho$ is supported on the unit ball, it can be checked that the arguments of \cite[proof of Theorem 2.9]{villani2003topics} (in particular boundedness from below) carry over, when simply adding $|x|$ to the potential $\hat{f}$ instead of adding $|x|^2/2$ to $\hat{f}$ and $|y|^2/2$ to $\hat{g}$.}
Furthermore, as $\text{supp}(\rho)\subseteq B_1(0)$ we conclude $\|\partial \hat{f}\|_{\infty}\le 1$ by \eqref{eq:subdiff} and
\begin{align*}
C(\nu, \rho)-C(\mu, \rho)&\ge \int \hat{f}\,d\nu+\int \hat{g}\,d\rho
-\left(\int \hat{f}\,d\mu+ \int \hat{g}\,d\rho\right)\\
&=\int \hat{f}\,d\nu-\int \hat{f}\,d\mu\\
&\ge \inf_{f \in \mathcal{C}^1(\R^d)} \left(\int f\,d\nu-\int f\,d\mu\right).
\end{align*}
Taking the infimum over $\rho\in \mathcal{P}^1(\R^d)$ shows that 
\begin{align*}
\inf_{\rho\in \mathcal{P}^1(\R^d)} \left( C(\nu, \rho)- C(\mu,\rho)\right)\ge \inf_{f \in \mathcal{C}^1(\R^d)}\left(\int f\,d\nu - \int f\,d\mu \right).
\end{align*}
On the other hand, fix $f\in \mathcal{C}^1(\R^d)$ and set $g:=f^*$. Define $\hat{\rho}:=\nabla f_*\mu$ and note that $\hat{\rho}\in \mathcal{P}^1(\R^d)$. Then again by Brenier's theorem we obtain optimality of the pair $(f,g)$ for $C(\mu,\hat{\rho})$, and thus
\begin{align*}
\int f\,d\nu- \int f\,d\mu  
 &= \left( \int f\,d\nu+\int g\,d\hat{\rho}\right) -\left( \int g\,d\hat{\rho} +\int f\,d\mu\right)\\
&\ge C(\nu, \hat{\rho})- C(\mu, \hat{\rho})\\
&\ge \inf_{\rho\in \mathcal{P}^1(\R^d)} \left( C(\nu, \rho)- C(\mu, \rho)\right).
\end{align*}
Taking the infimum over $f\in \mathcal{C}^1(\R^d)$ shows
\begin{align*}
\inf_{f \in \mathcal{C}^1(\R^d)}\left(\int f\,d\nu - \int f\,d\mu \right)\ge \inf_{\rho\in \mathcal{P}^1(\R^d)} \left( C(\nu,\rho)-C(\mu,\rho)\right).
\end{align*}
This concludes the proof.
\end{proof}

\subsection{Proof of Theorem \ref{thm:main2}: the general case}

We now prove Theorem \ref{thm:main2} for general measures $\mu,\nu\in \mathcal{P}_1(\R^d)$ through approximation in the $1$-Wasserstein sense.

\begin{proof}[Proof of Theorem \ref{thm:main2}]
Let us take sequences of $(\mu_n)_{n\in \N}$, $(\nu_n)_{n\in \N}$ in $\mathcal{P}_1(\R^d)$ satisfying
\begin{align*}
\lim_{n\to \infty} \mathcal{W}_1(\mu, \mu_n)=0=\lim_{n\to \infty} \mathcal{W}_1(\nu, \nu_n), \qquad \mu_n, \nu_n\sim \lambda \text{ for all } n\in \N,
\end{align*}
where $\mathcal{W}_1$ denotes the $1$-Wasserstein distance.
Recall that $\mathcal{C}^1(\R^d)$ denotes the set of convex $1$-Lipschitz functions. Thus, e.g. by the Kantorovich-Rubinstein formula (\cite[(5.11)]{villani2008optimal}),
\begin{align}\label{eq:aux1}
\lim_{n\to \infty} \sup_{f\in \mathcal{C}^1(\R^d)} \left| \int f d\mu-\int f d\mu_n\right|\le \lim_{n\to \infty}\mathcal{W}_1(\mu,\mu_n)=0.
\end{align}
The same holds for $(\nu_n)_{n\in \N}$ and $\nu$. Next, take an optimal coupling $\pi=\pi(dx,dy)$ for $C(\nu,\rho)$ and an optimal coupling $\pi^n=\pi^n(dx,dz)$ for $\mathcal{W}_1(\nu,\nu_n)$. Then $\hat{\pi}^n(dy,dz):= \int \pi^n(dx,dz)\pi_x(dy)$ is a coupling of $\rho$ and $\nu_n$. Furthermore, as $|y|\le 1$ $\rho$-a.s. we have
\begin{align*}
\begin{split}
C(\nu,\rho)-C(\nu_n,\rho)&\le \left|\int \langle y,x-z\rangle \,\pi^n(dx,dz)\pi_x(dy)\right|\\
&\le \int |x-z|\,\pi^n(dx,dz) \\
&\le \mathcal{W}_1(\nu_n,\nu).
\end{split}
\end{align*}
Exchanging the roles of $\nu$ and $\nu_n$ then yields
\begin{align*}
|C(\nu,\rho)-C(\nu_n,\rho)|\le \mathcal{W}_1(\nu_n,\nu).
\end{align*}
As the rhs is independent of $\rho\in \mathcal{P}^1(\R^d)$ this shows
\begin{align}\label{eq:aux2}
\lim_{n\to \infty} \sup_{\rho\in \mathcal{P}^1(\R^d)}\left|C(\nu,\rho)-C(\nu_n,\rho)\right|=0.
\end{align}
A similar argument holds for $(\mu_n)_{n\in \N}$ and $\mu$. We can now write
\begin{align*}
\inf_{f\in \mathcal{C}^1(\R^d)}\left(\int f\,d\nu - \int f\,d\mu \right)
&=\inf_{f\in \mathcal{C}^1(\R^d)}\Bigg[\left(\int f\,d\nu_n - \int f\,d\mu_n\right) \\
&+ \left(\int f\,d\nu- \int f\,d\nu_n \right)
-\left(\int f\,d\mu - \int f\,d\mu_n \right)\Bigg]
\end{align*}
and 
\begin{align*}
\inf_{\rho\in \mathcal{P}^1(\R^d)} \big( C(\nu,\rho)-C(\mu,\rho)\big)
&= \inf_{\rho\in \mathcal{P}^1(\R^d)} \Bigg[ \big(C(\nu_n,\rho)-C(\mu_n,\rho)\big)\\
&+  \big( C(\nu,\rho)-C(\nu_n,\rho)\big)
-   \big( C(\mu,\rho)-C(\mu_n,\rho)\big).\Bigg]
\end{align*}
Applying Proposition \ref{prop:1}, taking $n\to \infty$ and using \eqref{eq:aux1}, \eqref{eq:aux2} then concludes the proof.
\end{proof}

\subsection{Proof of Corollary \ref{thm:main}} We now detail the proof of Corollary \ref{thm:main}. We start with a preliminary result, which is an immediately corollary of Theorem \ref{thm:main2}.

\begin{corollary}\label{cor:1}
Assume $\mu,\nu\in \mathcal{P}_1(\R^d)$.
Then we have
\begin{align}\label{eq:res}
\inf_{f\ \mathrm{ convex}}\left(\int f\,d\nu - \int f\,d\mu \right)=\inf_{\rho\in \mathcal{P}^\infty (\R^d)} \left( C(\nu,\rho)-C(\mu,\rho)\right),
\end{align}
where $\mathcal{P}^\infty(\R^d)$ denotes the set of probability measures with bounded support.
In particular
\begin{align*}
\int f\,d\mu \le  \int f\,d\nu \qquad\text{for all convex functions }f:\R^d\to \R
\end{align*}
if and only if
\begin{align*}
C(\mu,\rho)\le C(\nu,\rho)\qquad \text{for all }\rho\in \mathcal{P}^\infty(\R^d).
\end{align*}
\end{corollary}

\begin{proof}
Multiplying both sides of \eqref{eq:main} by $k>0$ yields
\begin{align*}
\inf_{f\in \mathcal{C}^k(\R^d)}\left(\int f\,d\nu - \int f\,d\mu \right)=\inf_{\rho\in \mathcal{P}^k(\R^d)} \left( C(\nu,\rho)-C(\mu,\rho)\right)
\end{align*}
with the definitions
\begin{align*}
\mathcal{P}^k(\R^d)=\{\rho\in \mathcal{P}(\R^d): \ \text{supp}(\rho)\subseteq B_k(0)\}
\end{align*}
and
\begin{align*}
\mathcal{C}^k(\R^d):=\{f:\R^d\to \R \text{ convex}, \|\partial f\|_{\infty}\le k\}.
\end{align*}
Taking $k\to \infty$ we obtain
\begin{align*}
\inf_{f\text{ convex, Lipschitz}}\left(\int f\,d\nu - \int f\,d\mu \right)=\inf_{\rho\in \mathcal{P}^\infty(\R^d)} \left( C(\nu,\rho)-C(\mu,\rho)\right).
\end{align*}
Lastly, any convex function $f:\R^d\to \R$ can be approximated pointwise from below by convex Lipschitz functions. Thus 
\begin{align*}
\inf_{f\ \mathrm{ convex}}\left(\int f\,d\nu - \int f\,d\mu \right)=\inf_{\rho\in \mathcal{P}^\infty(\R^d)} \left( C(\nu,\rho)-C(\mu,\rho)\right).
\end{align*}
The claim thus follows.
\end{proof}

\begin{remark}
If $\mu,\nu \in \mathcal{P}_p(\R^d)$ for some $p\ge 1$, then by H\"older's inequality and density of finitely supported measures in the $q$-Wasserstein space we also obtain 
\begin{align*}
\inf_{f\ \mathrm{ convex}}\left(\int f\,d\nu - \int f\,d\mu \right)=\inf_{\rho\in \mathcal{P}_q(\R^d)} \left( C(\nu,\rho)-C(\mu,\rho)\right),
\end{align*}
where $1/p+1/q=1$.
\end{remark}

\begin{proof}[Proof of Corollary \ref{thm:main}]
Recall from \eqref{eq:was} that
\begin{align*}
C(\mu,\rho)&= \frac{1}{2}\left( \int |x|^2\,\mu(dx)+\int |z|^2 \,\rho(dz) - \mathcal{W}_2(\mu,\rho)^2\right),\\
C(\nu,\rho)&= \frac{1}{2} \left(\int |y|^2 \,\nu(dy)+\int |z|^2\,\rho(dz) - \mathcal{W}_2(\nu,\rho)^2\right).
\end{align*}
Combining this with \eqref{eq:res} from Corollary \ref{cor:1} yields
\begin{align*}
&\inf_{f\text{ convex}}\left(\int f\,d\nu - \int f\,d\mu \right)=\inf_{\rho\in \mathcal{P}^\infty(\R^d)} \left( C(\nu,\rho)-C(\mu,\rho)\right)\\
&=\frac{1}{2} \inf_{\rho\in \mathcal{P}^\infty(\R^d)} \Big( \int |y|^2 \,\nu(dy)+\int |z|^2 \,\rho(dz) - \mathcal{W}_2(\nu,\rho)^2\\
&\qquad\qquad\qquad\qquad-\int |x|^2 \,\mu(dx)-\int |z|^2\,\rho(dz) + \mathcal{W}_2(\mu,\rho)^2\Big)\\
&=\frac{1}{2} \inf_{\rho\in \mathcal{P}^\infty(\R^d)} \Big(\mathcal{W}_2(\mu,\rho)^2 -\mathcal{W}_2(\nu,\rho)^2 +\int |y|^2\,\nu(dy) -\int|x|^2 \,\mu(dx) \Big).
\end{align*}
Thus
\begin{align*}
&\int f\,d\mu \le \int f\,d\nu \qquad \text{for all convex functions } f:\R^d\to \R\\
\Leftrightarrow &\inf_{f\text{ convex}}\left(\int f\,d\nu - \int f\,d\mu \right)\ge 0\\
\Leftrightarrow &\inf_{\rho\in \mathcal{P}^\infty(\R^d)} \Big(\mathcal{W}_2(\mu,\rho)^2 -\mathcal{W}_2(\nu,\rho)^2\Big)\ge \int |x|^2 \,\mu(dx)-\int |y|^2\,\nu(dy)\\
\Leftrightarrow &\sup_{\rho\in \mathcal{P}^\infty(\R^d)} \Big(\mathcal{W}_2(\nu,\rho)^2 -\mathcal{W}_2(\mu,\rho)^2\Big)\le \int |y|^2 \,\nu(dy)-\int |x|^2\,\mu(dx).
\end{align*}
The claim follows.
\end{proof}

\section{Numerical examples}\label{sec:num}

In this section we illustrate Theorem \ref{thm:main2} numerically. We focus on the following toy examples, where convex order or its absence is easy to establish:

\begin{example}\label{ex:1}
$\mu=\mathcal{N}(0, \sigma^2 I)$ and $\nu=\mathcal{N}(0, I)$ for $\sigma^2\in [0,2]$ for $d=1,2$.
\end{example}

\begin{example}\label{ex:2}
$\mu= \frac{1}{2} \left( \delta_{-1-s} +\delta_{1+s}\right)$ and $\nu=\frac{1}{2} \left( \delta_{-1}+\delta_{1} \right)$ for $s\in [-1,1]$.
\end{example}

\begin{example}\label{ex:3}
$$\mu= \frac{1}{4} \left( \delta_{(-1-s, 0)} +\delta_{(1+s, 0)}+ \delta_{(0, 1+s)}+ \delta_{(0,-1-s)}\right)$$ and 
$$\nu=\frac{1}{4}\left( \delta_{(-1, 0)} +\delta_{(1, 0)}+ \delta_{(0, 1)}+ \delta_{(0,-1)}\right)$$ for $s\in [-1,1]$.
\end{example}

A general numerical implementation for testing convex order of the two measures $\mu,\nu\in \mathcal{P}_2(\R^d)$ in general dimensions and the examples discussed here can be found in the Github repository \href{https://github.com/johanneswiesel/Convex-Order}{https://github.com/johanneswiesel/Convex-Order}. In the implementation we use the POT package (\href{https://pythonot.github.io}{https://pythonot.github.io}) to compute optimal transport distances. 

Let us set
\begin{align*}
V(\mu,\nu):=\inf_{\rho\in \mathcal{P}^1(\R^d)} \left( C(\nu,\rho)-C(\mu,\rho)\right)
\end{align*}
and note that by Theorem \ref{thm:main2} we have the relationship
\begin{align*}
\mu\preceq_c \nu \quad \Leftrightarrow \quad V(\mu,\nu)\ge 0.
\end{align*}
Clearly the computation of $V(\mu,\nu)$ hinges on the numerical exploration of the convex set of probability measures $\mathcal{P}^1(\R^d)$. We propose two methods for this: our first method only considers finitely supported measures $\rho$, which are dense in $\mathcal{P}^1(\R^d)$ in the Wasserstein topology. It relies on the \textit{Dirichlet distribution} on the space $R^{g-1}$, $g\in \N$, with density
\begin{align*}
f(x_1, \dots, x_g; \alpha_1, \dots, \alpha_g)=\frac{1}{B(\alpha)} \prod_{i=1}^g x_i^{\alpha_i-1}
\end{align*}
for $x_1, \dots, x_g\in [0,1]$ satisfying $\sum_{i=1}^g x_i=1$. Here $\alpha_1,\dots, \alpha_g>0$, $\alpha:=(\alpha_1, \dots, \alpha_g)$ and $B(\alpha)$ denotes the Beta function. Fixing $g$ grid points $\{k_1, \dots, k_g\}$ in $B_1(0)$, we can consider any realization of a Dirichlet random variable $(X_1, \dots, X_g)$ as a probability distribution assigning probability mass $X_i$ to the grid point $k_i$, $i\in \{1, \dots, g\}$. This leads to the following algorithm:

\begin{algorithm}
\caption{Basic algorithm for Indirect Dirichlet method}\label{alg:1}
\begin{algorithmic}
\Require{probability measures $\mu$, $\nu$, maximal number of evaluations $N$, number of grid points $g$.}
\Ensure{$V(\mu,\nu)$}
\Statex
Generate a grid $G$ of $B_1(0)$ of $g$ equidistant points and consider Dirichlet random variables modelling $\rho$ supported on $G$. Use Bayesian optimization to solve
\begin{align*}
\inf \,[C(\rho, \nu)-C(\rho,\mu)]
\end{align*}
over the set of Dirichlet distributions on $\R^{g-1}$. Terminate after $N$ steps.
\State \Return  $\inf \,C(\rho, \nu)-C(\rho,\mu).$
\end{algorithmic}
\end{algorithm}

The main computational challenge in Algorithm \ref{alg:1} is the efficient evaluation of $C(\rho, \nu)$ and $C(\rho,\mu)$. For this we aim to write $C(\rho,\nu)$ and $C(\rho,\mu)$ as linear programs. We offer two different variants of Algorithm \ref{alg:1}:

\begin{itemize}
\item \textit{Indirect Dirichlet method with histograms:} If we have access to finitely supported approximations $\textbf{a}$ and $\textbf{b}$ of $\mu$ and $\nu$ respectively and the measure $\rho$ is supported on $G$ as above, then we solve the linear programs $C(\textbf{a},\rho)$ and $C(\textbf{b}, \rho)$ as is standard in optimal transport theory.
\item \textit{Indirect Dirichlet method with samples:} here we draw a number of samples from $\mu$ and $\nu$ respectively and denote the respective empirical distributions of these samples by $\textbf{a}$ and $\textbf{b}$. As before we assume that we have access to a probability measure $\rho$ supported on $G$. We then solve the linear programs $C(\textbf{a},\rho)$ and $C(\textbf{a}, \rho)$. 
\end{itemize}

An alternative to Algorithm \ref{alg:1} is to directly draw samples from a distribution $\rho \in \mathcal{P}^1(\R^d)$. We call this the \textit{Direct randomized Dirichlet method}, see Algorithm \ref{alg:2} below.\\

\begin{algorithm}[H]
\caption{Direct randomized Dirichlet method}\label{alg:2}
\begin{algorithmic}
\Require{probability measures $\mu$, $\nu$, maximal number of evaluations $N$.}
\Ensure{$V(\mu,\nu)$}
\Statex
Draw samples from $\mu$ and $\nu$ and denote the empirical distributions of these samples by $\textbf{a}$ and $\textbf{b}$ respectively. Draw samples from a Dirichlet distribution and randomize their signs, under the constraint that the empirical distribution $\rho$ of these samples is an element of $\mathcal{P}^1(\R^d)$.
Use Bayesian optimization to solve
\begin{align*}
\inf \,[C(\rho, \nu)-C(\rho,\mu)]
\end{align*}
over the set of these distributions. Terminate after $N$ steps.
\State \Return minimal value of $\inf\,[C(\rho, \mu)-C(\rho,\nu)].$
\end{algorithmic}
\end{algorithm}

We refer to the github repository for a more detailed discussion, in particular for the implementation and further comments. For each example stated at the beginning of this section and each pair $(\mu,\nu)$ we plot $V(\mu,\nu)$ for the three methods discussed above, see Figures \ref{fig:1} and \ref{fig:2}.

\begin{figure}[h!]
\begin{center}
\begin{minipage}{0.44\textwidth}
\includegraphics[scale=0.22]{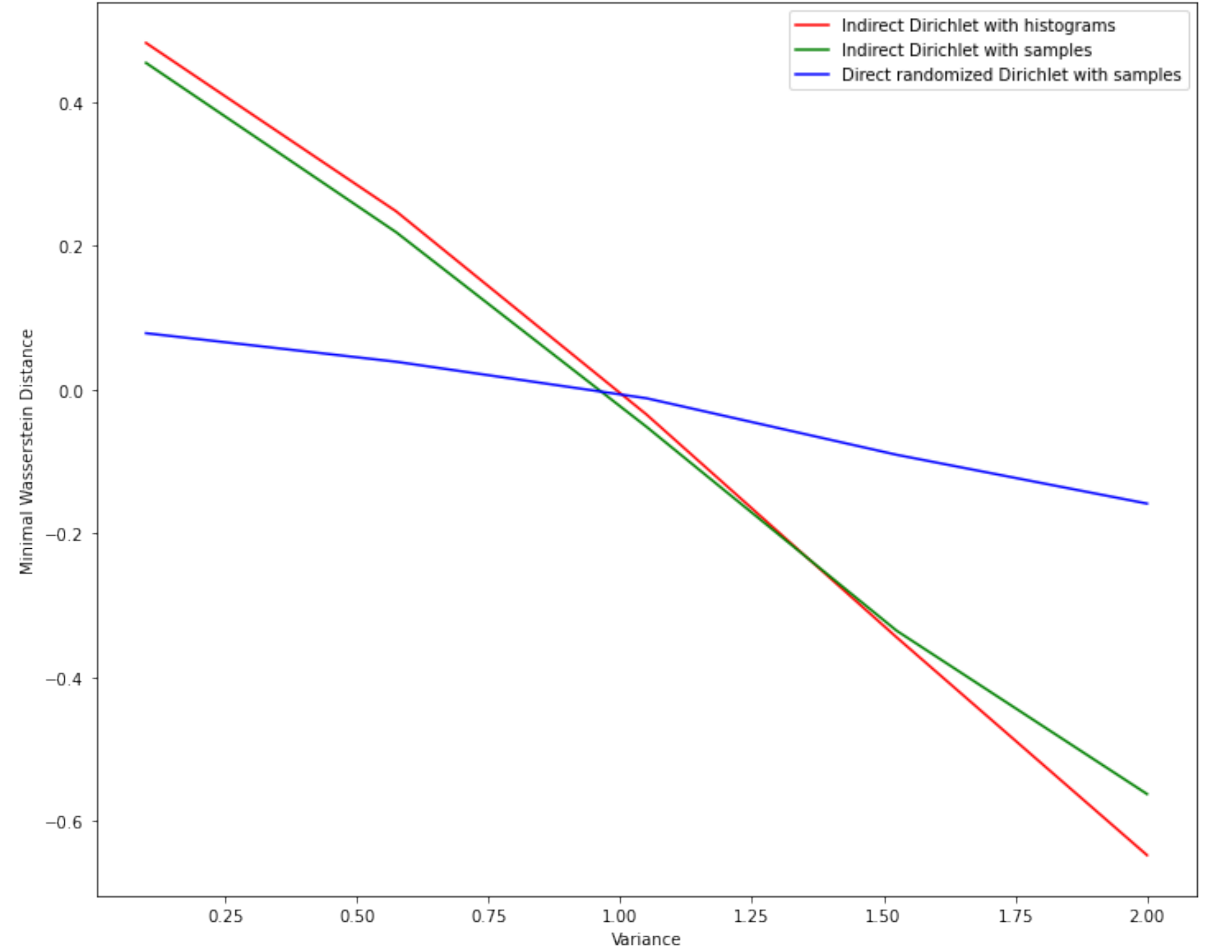}
\end{minipage}
\quad
\begin{minipage}{0.44\textwidth}
\includegraphics[scale=0.22]{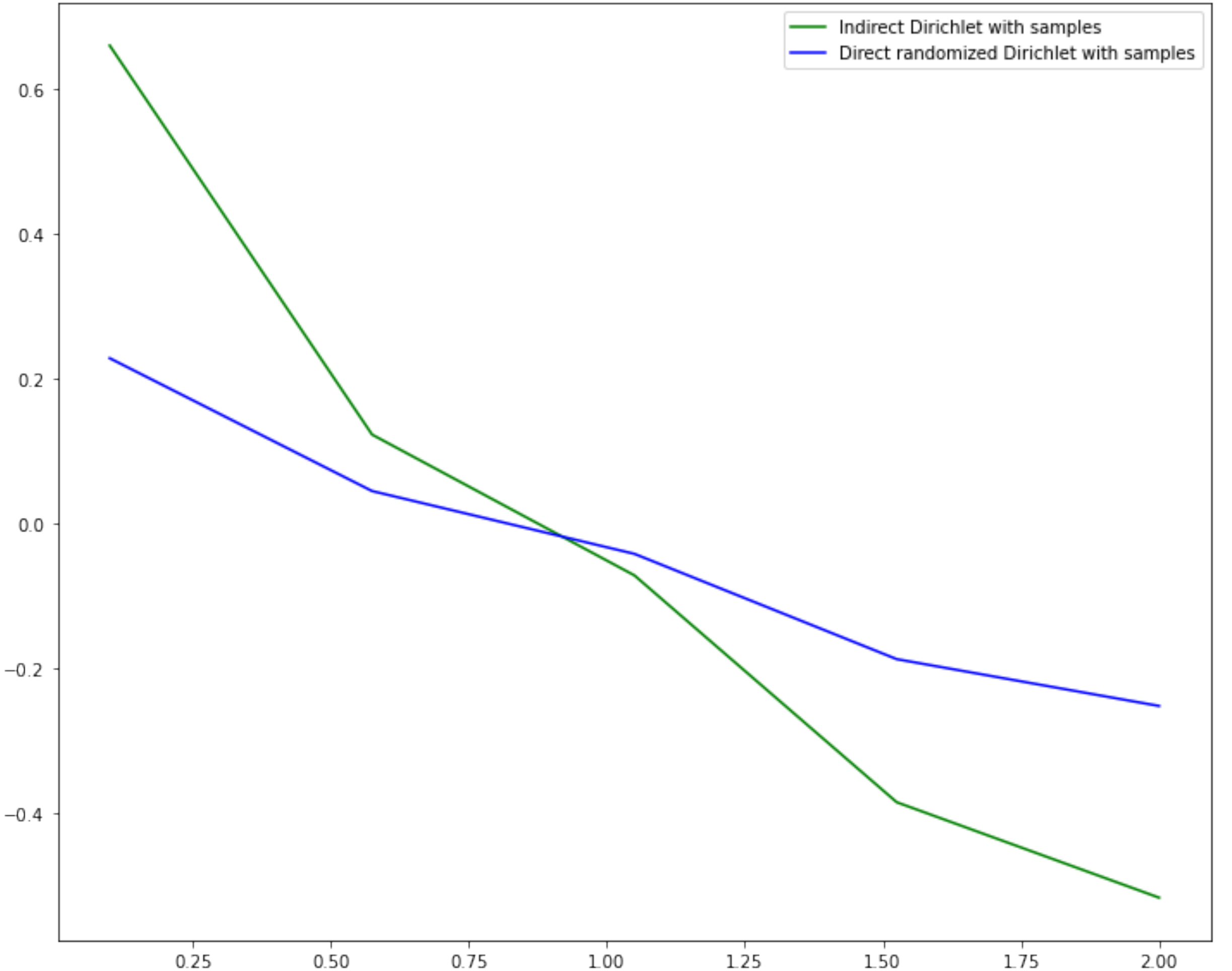}
\end{minipage}
\caption{Values of different estimators of $V(\mu,\nu)$ plotted against $\sigma$ for Example \ref{ex:1}. Both plots use $N=100$ samples.}\label{fig:1}
\end{center}
\end{figure}

\begin{figure}[h!]
\begin{center}
\begin{minipage}{0.44\textwidth}
\includegraphics[scale=0.22]{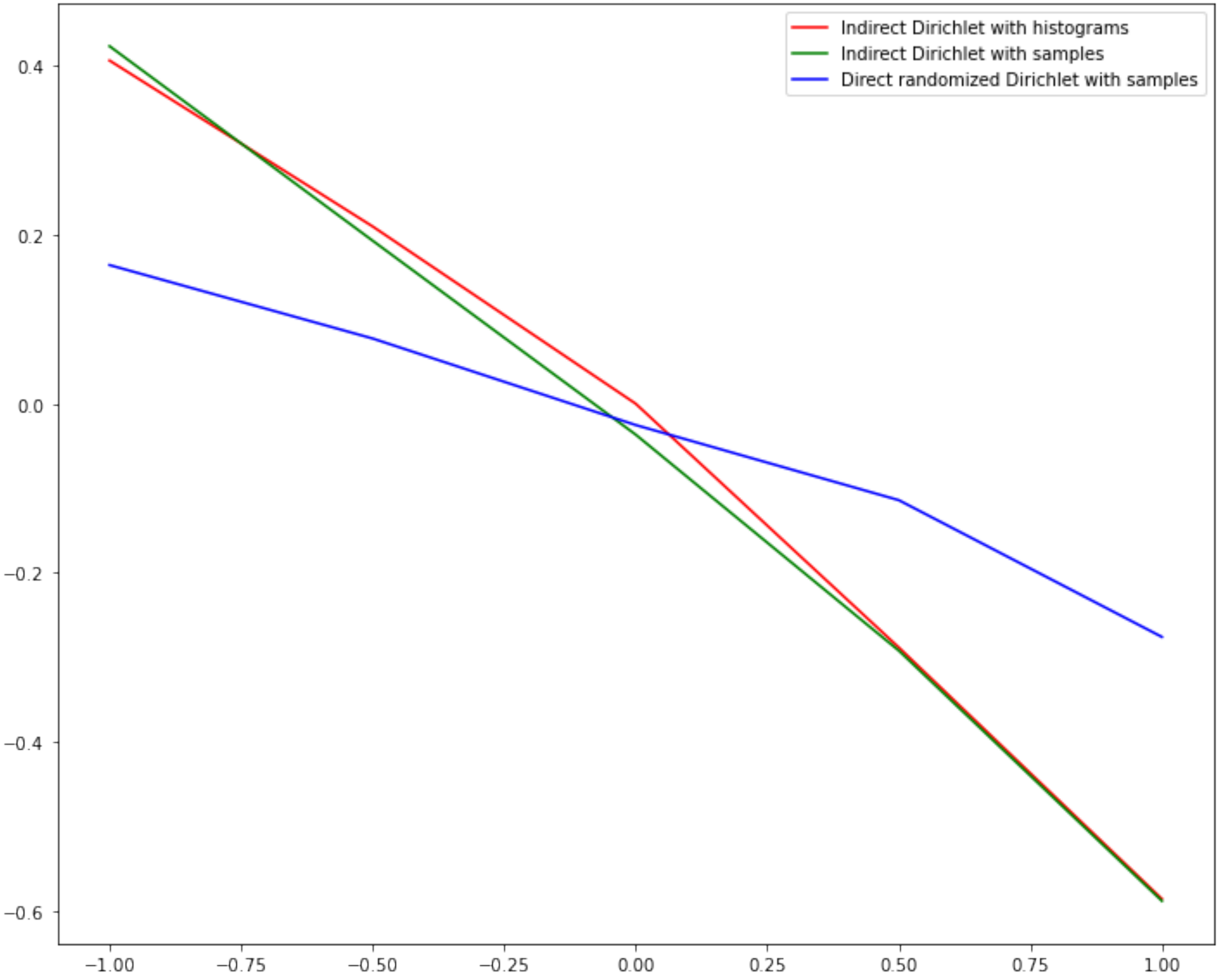}
\end{minipage}
\quad
\begin{minipage}{0.44\textwidth}
\includegraphics[scale=0.22]{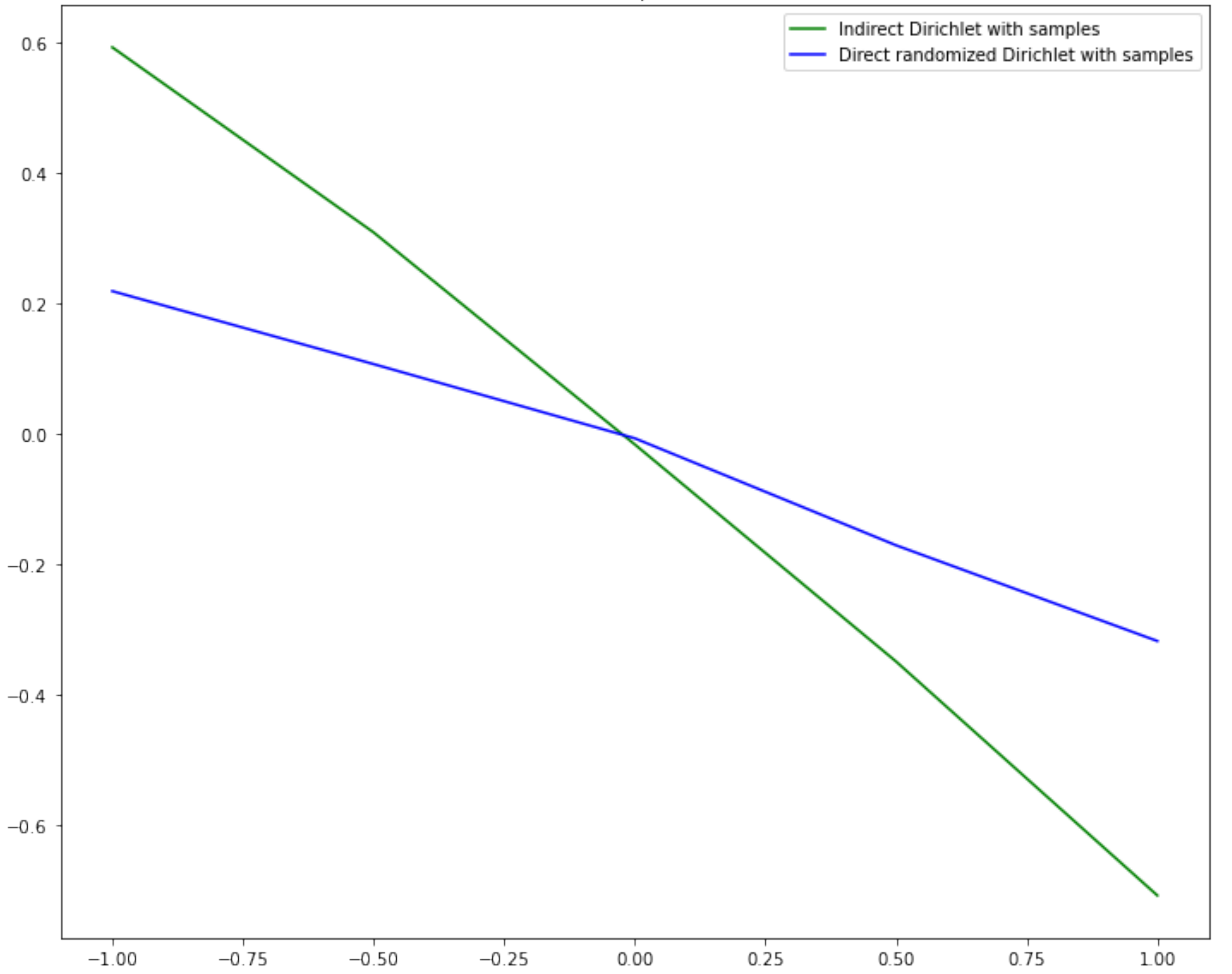}
\end{minipage}
\caption{Values of different estimators of $V(\mu,\nu)$ plotted against $s$ for Example \ref{ex:2} (left) and \ref{ex:3} (right). Both plots use $N=100$ samples.}\label{fig:2}
\end{center}
\end{figure}

Discounting numerical errors, all estimators seem to detect convex order. The direct randomized Dirichlet method is less complex; however it does not seem to explore the $\mathcal{P}^1(\R^d)$-space as well as the two indirect Dirichlet methods. On the other hand, both of the indirect Dirichlet methods yield very similar results for the examples considered. As the name suggests, the ``indirect Dirichlet method with samples" works on samples directly, which might be more convenient for practical applications on real data.

As can be expected from the numerical implementation, the histogram method consistently yields the lowest runtimes, while runtimes of the other methods are much higher. Indeed, when working with samples, the weights of the empirical distributions are constant, while the OT cost matrices $\textbf{M}_{\textbf{a}}$ and $\textbf{M}_{\textbf{b}}$ in the implementation have to re-computed in each iteration and this is very costly; for the histogram method, the weights $\rho$ change, while the grid stays constant --- and thus also $\textbf{M}_{\textbf{a}}$ and $\textbf{M}_{\textbf{b}}$ .

\section{Model independent arbitrage strategies}\label{sec:arbitrage}

Let us consider a financial market with $d$ financial assets and denote its price process by $(S_t)_{t\ge 0}$. Let us assume $S_0=s_0\in \R$ and fix two maturities $T_1<T_2$. If call options with these maturities are traded at all strikes, then the prices of the call options determine the distribution of $S_{T_1}$ and $S_{T_2}$ under any martingale measure; this fact was first established by \cite{breeden1978prices}. Let us denote the laws of $S_{T_1}$ and $S_{T_2}$ by $\mu$ and $\nu$  respectively. If trading is only allowed at $0,T_1$ and $T_2$, the following definition is natural and will be crucial for our analysis.

\begin{definition}
The triple of measurable functions $(u_1, u_2, \Delta)$ is a model-independent arbitrage if $u_1\in L^1(\mu)$, $u_2\in L^1(\nu)$ and 
\begin{align*}
u_1(x)-\int u_1\,d\mu+u_2(y)-\int u_2\,d\nu+\Delta(x)(y-x)>0, \quad\text{for all }(x,y)\in \R^d\times \R^d.
\end{align*}
If no such strategies exist, then we call the market \emph{free of model-independent arbitrage}.
\end{definition}

In the above, $u_1$ and $u_2$ can be interpreted as payoffs of Vanilla options with market prices $\int u_1\,d\mu$ and $\int u_2\,d\nu$ respectively, while the term $\Delta(x)(y-x)$ denotes the gains or losses from buying $\Delta(x)$ assets at time $T_1$ and holding them until $T_2$.

The following theorem makes the connection between model-independent arbitrages and convex order of $\mu$ and $\nu$ apparent. It can essentially be found in \cite[Theorem 3.4]{guyon2017bounds}.

\begin{theorem}\label{thm:guyon}
The following are equivalent:
\begin{enumerate}[(i)]
\item The market is free of model-independent arbitrage.
\item $\mathcal{M}(\mu,\nu)\neq \emptyset$.
\item $\mu\preceq_c\nu$.
\end{enumerate}
In particular, if $\mu\npreceq_c\nu$, then there exists a convex function $f$, such that the triple $(-f(x), f(y), -g(x))$ is a model-independent arbitrage. Here $g$ is a measurable selector of the subdifferential of $f$.
\end{theorem}
The strategy $(-f(x), f(y), -g(x))$ is often called a \textit{calendar spread}. As our setting is not quite exactly covered by \cite[Theorem 3.4]{guyon2017bounds} and the proof is not hard, we include it here.
\begin{proof}[Proof of Theorem \ref{thm:guyon}]
(ii)$\Leftrightarrow$(iii) is Strassen's theorem, see \cite{strassen1965existence}. If $\mu\npreceq_c\nu$, then by definition there exists a convex function $f$ such that $$\int f\,d\mu>\int f\,d\nu.$$ On the other hand, $f$ is convex and thus satisfies
\begin{align*}
f(y)-f(x)\ge g(x)(y-x) \quad \text{for all }(x,y)\in \R^d\times \R^d.
\end{align*}
Combining the two equations above shows that $(-f(x), f(y), -g(x))$ is a model-independent arbitrage, and thus (i)$\Rightarrow$(iii). It remains to show (ii)$\Rightarrow$(i), which is well known. Indeed, taking expectations in the inequality 
\begin{align*}
u_1(x)-\int u_1\,d\mu+u_2(y)-\int u_2\,d\nu+\Delta(x)(y-x)>0, \quad\text{for all }(x,y)\in \R^d\times \R^d
\end{align*}
under any martingale measure with marginals $\mu,\nu$ leads to a contradiction. This concludes the proof.
\end{proof}

As a direct consequence of Theorem \ref{thm:guyon}, we can use Theorem \ref{thm:main2} to detect model-independent arbitrages in the market under consideration: indeed, Theorem \ref{thm:main2} states that  $\mu\npreceq_c\nu$ implies existence of a probability measure $\rho\in \mathcal{P}^1(\R^d)$ satisfying 
\begin{align*}
C(\rho, \nu)-C(\rho,\mu)<0.
\end{align*}
Next, if $\nu\sim \lambda$, then the proof of Theorem \ref{thm:main2} shows that  $\rho=\nabla \hat{f}(x)_*\nu$ for some convex function $\hat{f}:\R^d\to \R$ and
\begin{align*}
\int \hat{f}\,d\nu-\int \hat{f}\,d\mu\le C(\rho, \nu)-C(\rho,\mu)<0, \quad \text{i.e. }\int \hat{f}\,d\nu<\int \hat{f}\,d\mu.
\end{align*}
In particular, a model-independent arbitrage strategy is given by calendar spread $(-\hat{f}(x), \hat{f}(x),-\nabla \hat{f}(x))$. Via an approximation argument, this result remains true for arbitrary probability measures $\nu\in \mathcal{P}_1(\R^d)$. In particular, we can use the same methods as in Section \ref{sec:num} to find $\rho$. We then estimate $\nabla \hat f(x)$ from the optimizing transport plan $\pi\in \Pi(\rho,\nu)$ of $C(\rho, \nu)$ by taking the conditional expectation $\int x\,\pi_y(dx)$, where $(\pi_y)_{y\in \R^d}$ denotes the conditional probability distribution of $\pi$ with respect to its second marginal $\nu$. This is a standard technique (see e.g. \cite{deb2021rates} for details). In conclusion we can obtain an explicit arbitrage strategy.

To illustrate the ideas outlined above, we return to Example \ref{ex:1}, i.e. $\mu=\mathcal{N}(0, \sigma^2 I)$ and $\nu=\mathcal{N}(0, I)$ for $\sigma^2>0$ and $d=1,2$. Having determined $\rho$ such that $C(\rho, \nu)-C(\rho,\mu)<0$, we estimate $\nabla \hat{f}$ numerically. We show estimates for $\nabla \hat{f}$ and $\hat{f}$ in the plots below.

\begin{figure}[h!]
\begin{center}
\begin{minipage}{0.44\textwidth}
\includegraphics[scale=0.4]{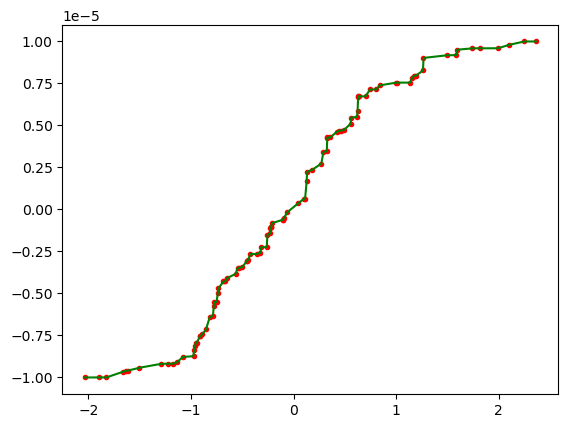}
\end{minipage}
\quad
\begin{minipage}{0.44\textwidth}
\includegraphics[scale=0.38]{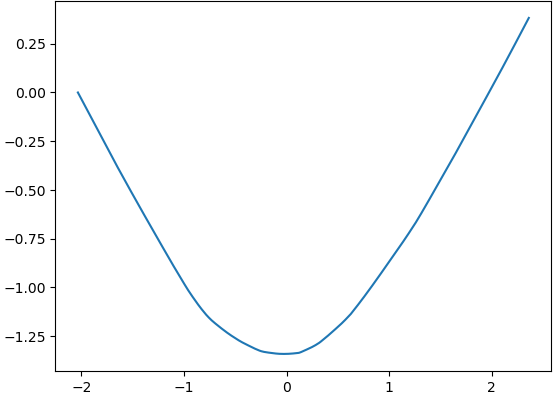}
\end{minipage}
\caption{Plot of estimates for $\nabla\hat{f}$ and $\hat{f}$ for $\mu=\mathcal{N}(0, 2), \nu=\mathcal{N}(0, 1)$, $d=1$. Both plots use $N=100$ samples.}\label{fig:1}
\end{center}
\end{figure}

\begin{figure}[h!]
\begin{center}
\includegraphics[scale=0.22]{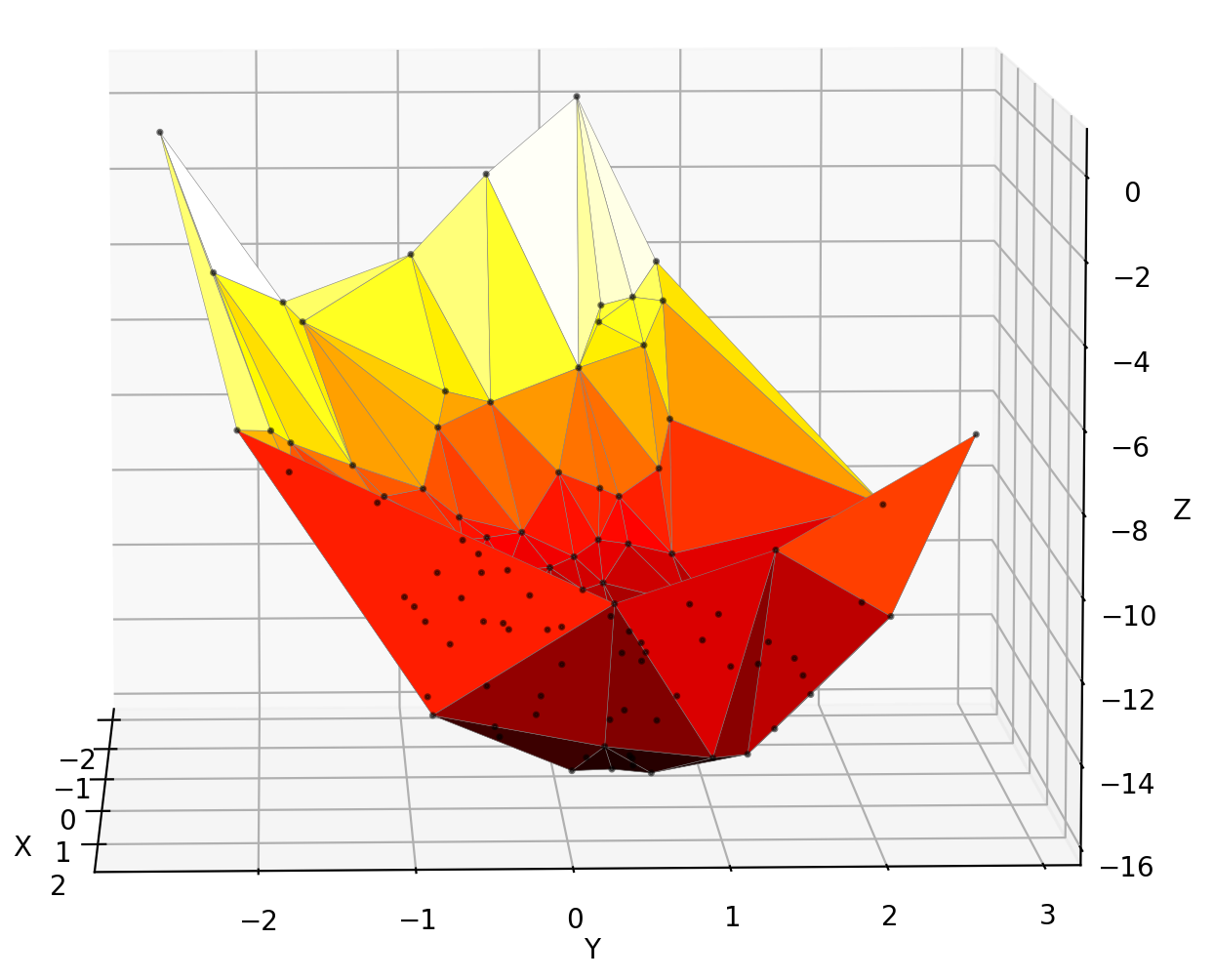}
\caption{Plot of estimate for $\hat{f}$ for $\mu=\mathcal{N}(0, 4I), \nu=\mathcal{N}(0, I)$, $d=2$. Both plots use $N=100$ samples.}\label{fig:2}
\end{center}
\end{figure}

\section{Remaining proofs}\label{sec:rem}

\begin{proof}[Proof of Corollary \ref{cor:c-concave}]
Recall that a function $g:\R^d\to \R$ is $c$-concave, iff $f(x):=|x|^2/2-g(x)$ is convex. In particular 
\begin{align*}
\int g\,d\mu - \int g\,d\nu &=\left(-\int \left[\frac{|x|^2}{2}-g(x)\right]\,\mu(dx)  +\int \left[\frac{|y|^2}{2}-g(y)\right]\,\nu(dy) \right)\\
&\qquad+\int \frac{|x|^2}{2}\,\mu(dx)-\int \frac{|y|^2}{2}\,\nu(dy)\\
&=\int f\,d\nu - \int f\,d\mu +\int \frac{|x|^2}{2}\,\mu(dx)-\int \frac{|y|^2}{2}\,\nu(dy).
\end{align*}
By \eqref{eq:res} we obtain
\begin{align*}
&\inf_{g\ c\text{-concave}}\left(\int g\,d\mu - \int g\,d\nu \right)\\
&=\inf_{f\text{ convex}}\left(\int f\,d\nu - \int f\,d\mu \right)+\int \frac{|x|^2}{2}\,\mu(dx)-\int \frac{|y|^2}{2}\,\nu(dy)\\
&=\frac{1}{2} \inf_{\rho\in \mathcal{P}^\infty(\R^d)} \Big(\mathcal{W}_2^2(\mu,\rho) -\mathcal{W}_2^2(\nu,\rho) +\int |y|^2\,\nu(dy) -\int|x|^2 \,\mu(dx)\\
&\qquad\qquad\qquad +\int |x|^2\,\mu(dx)-\int |y|^2\,\nu(dy)\Big)\\
&= \frac{1}{2} \inf_{\rho\in \mathcal{P}^\infty(\R^d)} \Big( \mathcal{W}_2(\mu,\rho)^2 - \mathcal{W}_2(\nu,\rho)^2\Big).
\end{align*}
This concludes the proof.
\end{proof}

\begin{proof}[Proof of Corollary \ref{cor:easy}]
The first claim follows from Corollary \ref{thm:main} by setting $\rho=\mu$. By \eqref{eq:was} the above implies
\begin{align*}
2\int |x|^2\,\mu(dx)\le 2\sup_{\pi\in \Pi(\mu,\nu)} \int \langle x,y\rangle \,\pi(dx,dy),
\end{align*}
so the second claim follows.
\end{proof}

\begin{proof}[Proof of Corollary \ref{cor:d1}]
First, \citep[Theorem 2 \& Lemma 1]{wang2020distortion} show that $\mu\preceq_c\nu$ iff
\begin{align}\label{eq: ruodu}
\int_0^1 [F_\nu^{-1}(1-u)-F_\mu(1-u)]\,dh(u)\ge 0
\end{align}
for all concave functions $h$ such that the above integral is finite. As any concave function is Lebesgue-almost surely differentiable, standard approximation arguments imply that \eqref{eq: ruodu} holds iff
\begin{align*}
\int_0^1 g(u)[F_\nu^{-1}(u)-F_\mu^{-1}(u)]\,du\ge 0
\end{align*}
for all bounded increasing left-continuous functions $g:(0,1)\to \R$. But 
\begin{align*}
\{F_\rho^{-1}: \ \rho \in \mathcal{P}(\R) \text{ with bounded support}\} 
\end{align*}
is exactly the set of all bounded increasing left-continuous functions on $(0,1)$. Noting that by \cite[Equation (2.47)]{villani2003topics}
\begin{align*}
\mathcal{W}_2(\nu, \rho)^2&=\int_0^1 (F_{\nu}^{-1}(x)-F_\rho^{-1}(x))^2\,dx\\
&= \int y^2\,\nu(dy)-2\int_0^1 F_{\nu}^{-1}(x) F_{\rho}^{-1}(x)\,dx+\int z^2\,\rho(dz),
\end{align*}
we calculate
\begin{align*}
\mathcal{W}_2(\nu, \rho)^2-\mathcal{W}_2(\mu, \rho)^2 &=\int y^2\,\nu(dy)-2\int_0^1 F_\rho^{-1}(u) F_\nu^{-1}(u) \,du
+\int z^2\,\rho(dz)\\
&\quad -\int x^2\,\mu(dy) +2\int_0^1  F_\rho^{-1}(u) F_\mu^{-1}(u) \,du -\int z^2\,\rho(dz)\\
&= 2\int_0^1  F_\rho^{-1}(u)[F_\mu^{-1}(u)-F_\nu^{-1}(u)]\,du \\
&\quad +\int y^2\,\nu(dy)-\int x^2\,\mu(dy).
\end{align*}
This concludes the proof.
\end{proof}

\bibliographystyle{plainnat}
\bibliography{bib}

\begin{thebibliography}{36}
\providecommand{\natexlab}[1]{#1}
\providecommand{\url}[1]{\texttt{#1}}
\expandafter\ifx\csname urlstyle\endcsname\relax
  \providecommand{\doi}[1]{doi: #1}\else
  \providecommand{\doi}{doi: \begingroup \urlstyle{rm}\Url}\fi

\bibitem[Acciaio et~al.(2013)Acciaio, Beiglb{\"o}ck, Penkner, and
  Schachermayer]{acciaio2013model}
B.~Acciaio, M.~Beiglb{\"o}ck, F.~Penkner, and W.~Schachermayer.
\newblock A model-free version of the {F}undamental {T}heorem of {A}sset
  {P}ricing and the {S}uper-replication {T}heorem.
\newblock \emph{Math. Finance}, DOI: 10.1111/mafi.12060, 2013.

\bibitem[Alfonsi and Jourdain(2020)]{alfonsi2020squared}
A.~Alfonsi and B.~Jourdain.
\newblock {Squared quadratic Wasserstein distance: optimal couplings and Lions
  differentiability}.
\newblock \emph{ESAIM Prob. Stat.}, 24:\penalty0 703--717, 2020.

\bibitem[Alfonsi et~al.(2019)Alfonsi, Corbetta, and
  Jourdain]{alfonsi2017sampling}
A.~Alfonsi, J.~Corbetta, and B.~Jourdain.
\newblock Sampling of one-dimensional probability measures in the convex order
  and computation of robust option price bounds.
\newblock \emph{Int. J. Theor. Appl. Finance}, 22\penalty0 (3), 2019.

\bibitem[Alfonsi et~al.(2020)Alfonsi, Corbetta, and
  Jourdain]{alfonsi2020sampling}
A.~Alfonsi, J.~Corbetta, and B.~Jourdain.
\newblock {Sampling of probability measures in the convex order by Wasserstein
  projection}.
\newblock \emph{Ann. Henri Poincare}, 56\penalty0 (3):\penalty0 1706--1729,
  2020.

\bibitem[Arnold(2012)]{arnold2012majorization}
B.~Arnold.
\newblock \emph{Majorization and the Lorenz order: A brief introduction},
  volume~43.
\newblock Springer Science \& Business Media, 2012.

\bibitem[Beiglb{\"o}ck et~al.(2013)Beiglb{\"o}ck, Henry-Labord{\`e}re, and
  Penkner]{beiglbock2013model}
M.~Beiglb{\"o}ck, P.~Henry-Labord{\`e}re, and F.~Penkner.
\newblock Model-independent bounds for option prices---a mass transport
  approach.
\newblock \emph{Finance Stoch.}, 17\penalty0 (3):\penalty0 477--501, 2013.

\bibitem[Beiglb{\"o}ck et~al.(2015)Beiglb{\"o}ck, Nutz, and
  Touzi]{beiglbock2015complete}
M.~Beiglb{\"o}ck, M.~Nutz, and N.~Touzi.
\newblock Complete duality for martingale optimal transport on the line.
\newblock \emph{Ann. Prob.}, 45\penalty0 (5):\penalty0 3038--3074, 2015.

\bibitem[Bernard et~al.(2017)Bernard, R{\"u}schendorf, and
  Vanduffel]{bernard2017value}
C.~Bernard, L.~R{\"u}schendorf, and S.~Vanduffel.
\newblock Value-at-risk bounds with variance constraints.
\newblock \emph{J. Risk Insur.}, 84\penalty0 (3):\penalty0 923--959, 2017.

\bibitem[Breeden and Litzenberger(1978)]{breeden1978prices}
D.~Breeden and R.~Litzenberger.
\newblock Prices of state-contingent claims implicit in option prices.
\newblock \emph{Journal of Business}, pages 621--651, 1978.

\bibitem[Brenier(1991)]{brenier1991polar}
Y.~Brenier.
\newblock Polar factorization and monotone rearrangement of vector-valued
  functions.
\newblock \emph{Commu. Pure Appl. Math.}, 44\penalty0 (4):\penalty0 375--417,
  1991.

\bibitem[Carlier(2008)]{carlier2008remarks}
Guillaume Carlier.
\newblock Remarks on toland's duality, convexity constraint and optimal
  transport.
\newblock \emph{Pacific Journal of Optimization}, 4\penalty0 (3):\penalty0
  423--432, 2008.

\bibitem[De~March and Touzi(2019)]{de2019irreducible}
H.~De~March and N.~Touzi.
\newblock Irreducible convex paving for decomposition of multidimensional
  martingale transport plans.
\newblock \emph{Ann. Prob.}, 47\penalty0 (3):\penalty0 1726--1774, 2019.

\bibitem[Deb et~al.(2021)Deb, Ghosal, and Sen]{deb2021rates}
Nabarun Deb, Promit Ghosal, and Bodhisattva Sen.
\newblock Rates of estimation of optimal transport maps using plug-in
  estimators via barycentric projections.
\newblock \emph{Advances in Neural Information Processing Systems},
  34:\penalty0 29736--29753, 2021.

\bibitem[Embrechts et~al.(2013)Embrechts, Puccetti, and
  R{\"u}schendorf]{embrechts2013model}
P.~Embrechts, G.~Puccetti, and L.~R{\"u}schendorf.
\newblock Model uncertainty and var aggregation.
\newblock \emph{J. Bank. Financ.}, 37\penalty0 (8):\penalty0 2750--2764, 2013.

\bibitem[Galichon et~al.(2014)Galichon, Henry-Labord{\`e}re, and
  Touzi]{galichon2014stochastic}
A.~Galichon, P.~Henry-Labord{\`e}re, and N.~Touzi.
\newblock A stochastic control approach to no-arbitrage bounds given marginals,
  with an application to lookback options.
\newblock \emph{Ann. Appl. Prob.}, 24\penalty0 (1):\penalty0 312--336, 2014.

\bibitem[Guo and Ob{\l}{\'o}j(2019)]{guo2019computational}
Gaoyue Guo and Jan Ob{\l}{\'o}j.
\newblock Computational methods for martingale optimal transport problems.
\newblock \emph{Ann. Appl. Prob.}, 29\penalty0 (6):\penalty0 3311--3347, 2019.

\bibitem[Guyon et~al.(2017)Guyon, Menegaux, and Nutz]{guyon2017bounds}
Julien Guyon, Romain Menegaux, and Marcel Nutz.
\newblock {Bounds for VIX futures given S\&P 500 smiles}.
\newblock \emph{Finance and Stochastics}, 21:\penalty0 593--630, 2017.

\bibitem[Jourdain and Margheriti(2022)]{jourdain2022martingale}
B.~Jourdain and W.~Margheriti.
\newblock {Martingale Wasserstein inequality for probability measures in the
  convex order}.
\newblock \emph{Bernoulli}, 28\penalty0 (2):\penalty0 830--858, 2022.

\bibitem[Kantorovich(1958)]{kantorovich1942translocation}
L.~Kantorovich.
\newblock On the translocation of masses.
\newblock \emph{Manag. Sci.}, \penalty0 (5):\penalty0 1--4, 1958.

\bibitem[Massa and Siorpaes(2022)]{pietro}
M.~Massa and P.~Siorpaes.
\newblock How to quantise probabilities while preserving their convex order.
\newblock \emph{arXiv preprint arXiv:2206.10514}, 2022.

\bibitem[Monge(1781)]{monge1781memoire}
G.~Monge.
\newblock \emph{M{\'e}moire sur la th{\'e}orie des d{\'e}blais et des
  remblais}.
\newblock De l'Imprimerie Royale, 1781.

\bibitem[M{\"u}ller and Stoyan(2002)]{muller2002comparison}
A.~M{\"u}ller and D.~Stoyan.
\newblock \emph{Comparison methods for stochastic models and risks}, volume
  389.
\newblock Wiley, 2002.

\bibitem[Nendel(2020)]{nendel2020note}
M.~Nendel.
\newblock A note on stochastic dominance, uniform integrability and lattice
  properties.
\newblock \emph{Bull. Lond. Math. Soc.}, 52\penalty0 (5):\penalty0 907--923,
  2020.

\bibitem[Ob{\l}{\'o}j and Siorpaes(2017)]{obloj2017structure}
J.~Ob{\l}{\'o}j and P.~Siorpaes.
\newblock Structure of martingale transports in finite dimensions.
\newblock \emph{arXiv preprint arXiv:1702.08433}, 2017.

\bibitem[Peyr{\'e} and Cuturi(2019)]{peyre2019computational}
G.~Peyr{\'e} and M.~Cuturi.
\newblock Computational optimal transport: With applications to data science.
\newblock \emph{Foundations and Trends{\textregistered} in Machine Learning},
  11\penalty0 (5-6):\penalty0 355--607, 2019.

\bibitem[Rachev and R{\"u}schendorf(1998)]{rachev1998mass}
S.~Rachev and L.~R{\"u}schendorf.
\newblock \emph{Mass Transportation Problems: Volume I: Theory}, volume~1.
\newblock Springer Science $\&$ Business Media, 1998.

\bibitem[Ross et~al.(1996)Ross, Kelly, Sullivan, Perry, Mercer, Davis,
  Washburn, Sager, Boyce, and Bristow]{ross1996stochastic}
S.~M Ross, J.~Kelly, R.~Sullivan, W.~Perry, D.~Mercer, R.~Davis, T.~Washburn,
  E.~Sager, J.~Boyce, and V.~Bristow.
\newblock \emph{Stochastic processes}, volume~2.
\newblock Wiley New York, 1996.

\bibitem[R{\"u}schendorf and Rachev(1990)]{ruschendorf1990characterization}
L.~R{\"u}schendorf and S.~Rachev.
\newblock {A characterization of random variables with minimum L2-distance}.
\newblock \emph{J. Multivariate Anal.}, 32\penalty0 (1):\penalty0 48--54, 1990.

\bibitem[R{\"u}schendorf and Uckelmann(2002)]{ruschendorf2002variance}
L.~R{\"u}schendorf and L.~Uckelmann.
\newblock Variance minimization and random variables with constant sum.
\newblock In et~al. Cuadras, editor, \emph{Distributions with given marginals
  and statistical modelling}, pages 211--222. Springer, 2002.

\bibitem[Shaked and Shanthikumar(2007)]{shaked2007stochastic}
M.~Shaked and J.~Shanthikumar.
\newblock \emph{Stochastic orders}.
\newblock Springer, 2007.

\bibitem[Strassen(1965)]{strassen1965existence}
V.~Strassen.
\newblock The existence of probability measures with given marginals.
\newblock \emph{Ann. Math. Statist.}, pages 423--439, 1965.

\bibitem[Tchen(1980)]{tchen1980inequalities}
A.~Tchen.
\newblock Inequalities for distributions with given marginals.
\newblock \emph{Ann. Prob.}, pages 814--827, 1980.

\bibitem[Villani(2003)]{villani2003topics}
C.~Villani.
\newblock \emph{Topics in optimal transportation}.
\newblock Number~58. American Mathematical Soc., 2003.

\bibitem[Villani(2008)]{villani2008optimal}
C.~Villani.
\newblock \emph{Optimal transport: old and new}, volume 338.
\newblock Springer Berlin, 2008.

\bibitem[Wang and Wang(2011)]{wang2011complete}
B.~Wang and R.~Wang.
\newblock The complete mixability and convex minimization problems with
  monotone marginal densities.
\newblock \emph{J. Multivariate Anal.}, 102\penalty0 (10):\penalty0 1344--1360,
  2011.

\bibitem[Wang et~al.(2020)Wang, Wang, and Wei]{wang2020distortion}
Q.~Wang, R.~Wang, and Y.~Wei.
\newblock Distortion riskmetrics on general spaces.
\newblock \emph{Astin Bull.}, 50\penalty0 (3):\penalty0 827--851, 2020.

\end{thebibliography}
\end{document}